\documentclass{amsart}
\usepackage{fullpage}
\usepackage{amsmath}
\usepackage{amsthm}
\usepackage{amssymb}
\usepackage{array}
\usepackage{delarray}
\usepackage[dvips]{graphics}
\usepackage{epsfig}
\usepackage{color}
\usepackage{fancyvrb}
\usepackage{float}
\usepackage{setspace}

\theoremstyle{plain}
\newtheorem{theorem}{Theorem}
\newtheorem*{theorem*}{Theorem}
\numberwithin{theorem}{section}
\newtheorem{proposition}[theorem]{Proposition}
\newtheorem{example}[theorem]{Example}
\newtheorem{lemma}[theorem]{Lemma}

\newtheorem{definition}[theorem]{Definition}
\newtheorem{remark}[theorem]{Remark}
\newtheorem{conjecture}[theorem]{Conjecture}

\theoremstyle{definition}

\theoremstyle{remark}

\begin{document}

\title{Bitangents of tropical plane quartic curves} 
\author{Matt Baker, Yoav Len, Ralph Morrison, Nathan Pflueger, Qingchun Ren}


\begin{abstract}
We study smooth tropical plane quartic curves and show that they satisfy certain properties analogous to (but also different from) smooth plane quartics in algebraic geometry. For example, we show that every such curve admits either infinitely many or exactly 7 bitangent lines.  We also prove that a smooth tropical plane quartic curve cannot be hyperelliptic. 
\end{abstract}

\maketitle

\begin{section}{Introduction}
In classical algebraic geometry, smooth plane quartics are the simplest examples of algebraic curves that are not hyperelliptic, as well as the simplest examples of canonically embedded curves.  The study of the enumerative geometry of these curves dates back to at least 1834, when Pl\"{u}cker \cite{Plucker} showed that a smooth plane quartic curve $C$ over $\mathbb{C}$ has $28$ bitangent lines.
Our main results in this paper are the following two theorems, which show that smooth tropical plane quartics behave  similarly to their algebraic counterparts:

\begin{theorem*}(\ref{bitangent_theorem})
Every smooth tropical plane quartic curve admits exactly 7 equivalence classes of bitangent lines.
\end{theorem*}

\begin{theorem*}(\ref{hyperelliptic_theorem})
Every smooth tropical plane quartic curve is non-hyperelliptic.
\end{theorem*}

See Definitions~\ref{def:bitangent} and \ref{def:equivalent_bitangents} for the precise meaning of bitangent lines and equivalence classes thereof, and see Definition \ref{hyperelliptic_definition} for the meaning of ``hyperelliptic'' in the tropical setting.

\medskip

Many classical arguments regarding smooth algebraic plane quartic curves break down in the tropical setting because smooth tropical quartics are not canonically embedded in the strongest possible sense: while the stable intersection with any tropical line gives a canonical divisor (cf. Lemma \ref{lineSectionsAreCanonicals}), it is {\bf not} the case that every canonical divisor arises in this way. For example, the standard representative of the canonical divisor of a smooth tropical plane quartic (when viewed as a metric graph) does not generally lie on a tropical line. Thus many standard algebro-geometric arguments 
are invalid tropically.  However, a combination of geometric results and case-by-case analysis will allow us to prove Theorems \ref{bitangent_theorem}
and Theorem \ref{hyperelliptic_theorem}.

\medskip

We make use of both the ``abstract'' and ``embedded'' points of view on tropical algebraic geometry in this paper.  For example, we employ the tropical Riemann-Roch theorem (a statement about abstract metric graphs), as well as Luo's theory of rank-determining sets of metric graphs and Zharkov's results on theta characteristics.  On the other hand, we also make use of Newton polygons, the balancing condition, and elementary facts from tropical intersection theory, all of which concern {\bf embeddings} of tropical curves in the plane.

\medskip

Although we do not study such questions in this paper, we are also interested in the relationship between algebraic curves and their tropicalizations.  If $K$ is an algebraically closed and non-trivially valued non-Archimedean field, the tropicalization map sends the points of a smooth plane curve $X=V(f)$ in $(K^*)^2$ into $\mathbb{R}^2$ by coordinate-wise valuation.  The Euclidean closure of the image of $X$ in $\mathbb{R}^2$ is denoted $\text{Trop}(X)$, and when endowed with an appropriate weight function it is a plane tropical curve as defined in Definition \ref{definition:tropical_curve}.  By a theorem of Kapranov, the tropical curve $\text{Trop}(X)$ is dual to the subdivision of the Newton polygon of $f$ induced by the valuations of the coefficients of $f$.  
When $X$ is a smooth plane quartic for which $\text{Trop}(X)$ is tropically smooth, it is natural to guess that the 28 bitangents on $X$ specialize in groups of 4. More precisely, we make the following conjecture: 
\begin{conjecture}
Let $X$ be a smooth plane quartic and assume that $\text{Trop}(X)$ is tropically smooth. Then each odd theta characteristic of $\Gamma$ is the specialization of four effective theta characteristics of $X$ (counted with multiplicity).  
\end{conjecture}

We remark that Chan and Jiradilok have recently confirmed the conjecture for curves whose skeleton is the first tropical curve pictured in Figure \ref{four_example_curves} \cite{CJ}.


\end{section}

\subsection*{Acknowledgements}
This paper arose from the first author's project group at the 2013 AMS Math Research Communities workshop on Tropical and Non-Archimedean Geometry in Snowbird, Utah. The authors are grateful to the AMS for their support of this program.  Thanks to Bernd Sturmfels for his encouragement of this project, and to Johannes Rau for helpful comments. We also thank the referees for their insightful remarks.  The first author was supported in part by the NSF grant DMS-1201473.

\begin{section}{Preliminaries}
We briefly review some of the basic definitions of tropical geometry that will be used in this paper. See \cite{Gathmann, FirstSteps} for more details.
 
\begin{definition}\label{definition:tropical_curve}
A \emph {plane tropical curve} is a planar graph $C$, together with positive integer weights on the edges, such that:
\begin{enumerate}
\item The slope of each edge is rational.
\item At each vertex, the following balancing condition holds: the weighted sum of the primitive integral vectors of the edges around the vertex is zero.
\end{enumerate}
A plane tropical curve is \emph{smooth of degree $d$} if it is dual to a unimodular triangulation of the triangle with vertices $(0,0), (d,0), (0,d)$.  Here, \emph{unimodular} means that it is subdivided to $d^2$ triangles each having area $\frac{1}{2}$.
\end{definition}

There are alternate definitions of plane tropical curves, including those that live in \emph{modified tropical planes} \cite{MikApp}.  Here, however, we use the above definition, so all plane tropical curves are subsets of $\mathbb{R}^2$.
In particular, a smooth plane tropical curve of degree $d$ is a trivalent graph with $d$ infinite edges pointing in each of the directions $(1,1), (-1,0), (0,-1)$, with
the weight of each edge equal to $1$. See Figure \ref{four_example_curves} for some examples with $d=4$. 
 Although plane tropical curves are unbounded subsets of $\mathbb{R}^2$, much of the relevant data is contained in a compact subset of the curve, which we now introduce.

\begin{definition}  The \emph {skeleton} of a plane tropical curve is the subgraph obtained by contracting all the leaf edges.
\end{definition}

To distinguish between a tropical line and a standard line in $\mathbb{R}^2$, we will refer to a standard line as \emph{Euclidean}. 

\begin{subsection}{Smooth tropical plane quartics}\label{quartics}
A smooth tropical plane quartic curve is obtained as the dual graph of a unimodular triangulation of the standard triangle $T$ with vertices $(0,0), (0,4), (4,0)$. Each interior lattice point gives rise to a cycle in the dual graph, and the standard triangle contains three integer points in its interior, so the genus of the tropical curve (i.e., the genus of its skeleton considered as a metric graph) is three. 

\begin{figure}[hbt]
\begin{center}
\includegraphics[
height=1.4in
]{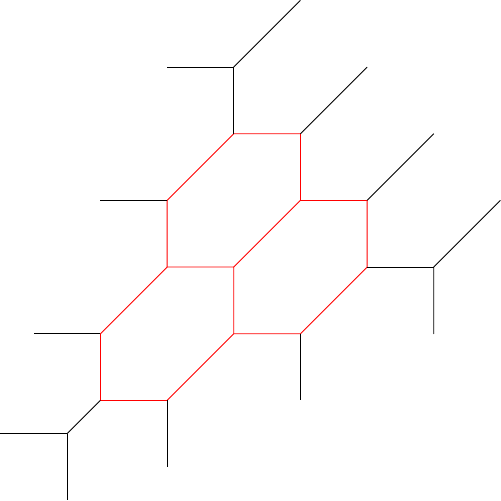}
\includegraphics[
height=1.4in
]{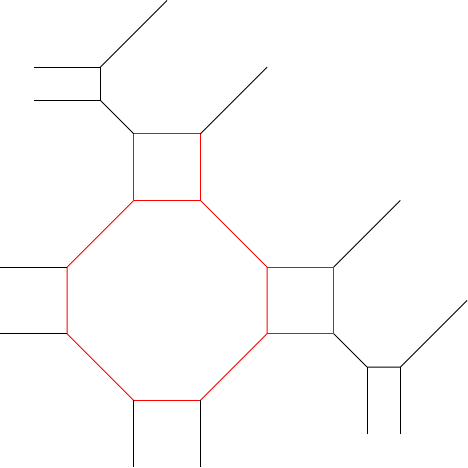}
\includegraphics[
height=1.4in
]{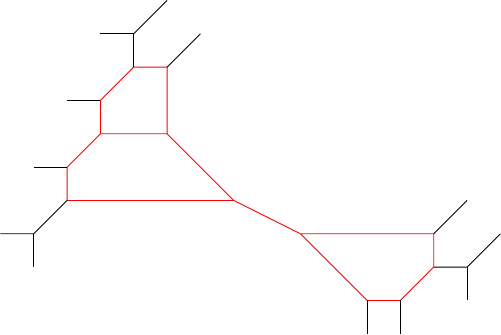}
\includegraphics[
height=1.4in
]{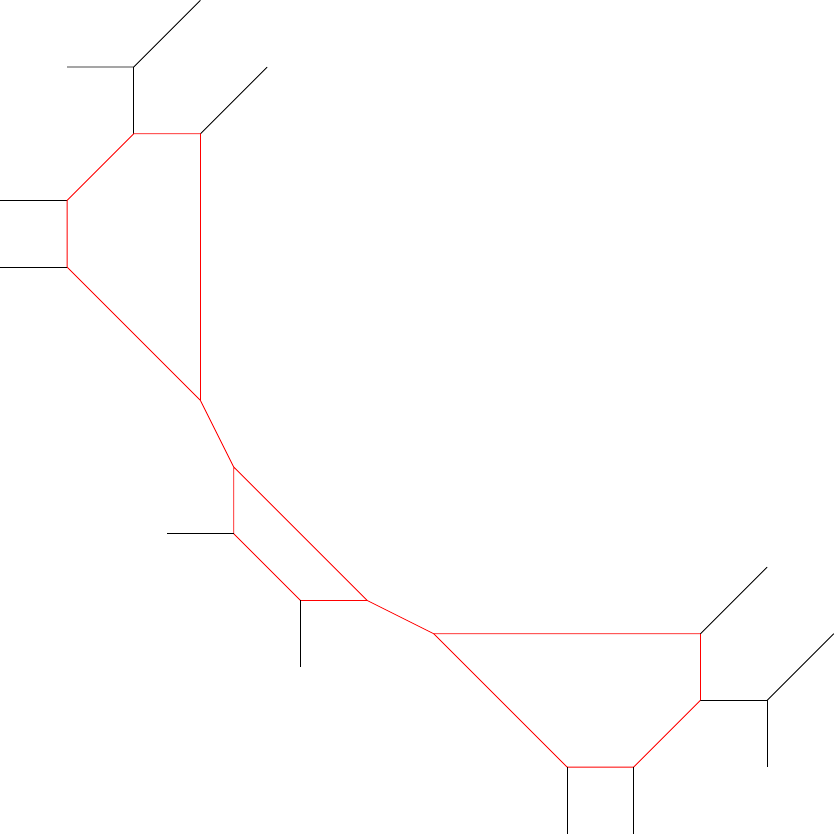}
\newline\includegraphics[
height=1in
]{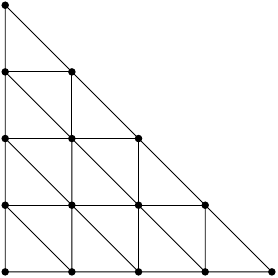}
\includegraphics[
height=1in
]{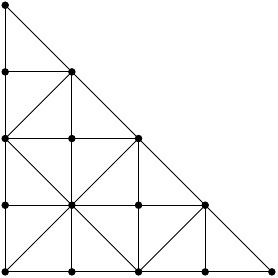}
\includegraphics[
height=1in
]{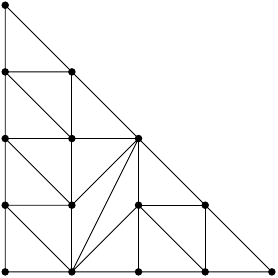}
\includegraphics[
height=1in
]{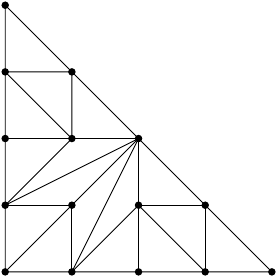}
\end{center}
\caption{Smooth tropical plane quartic curves and the dual triangulations.  From left to right, we have instances of tropical curves which we will refer to as the honeycomb, Mickey Mouse, one-bridge, and two-bridge. The skeletons are colored in red.}
\label{four_example_curves}
\end{figure}

As we will see in Proposition \ref{types}, there are four basic types of smooth tropical plane curves, examples of which are shown in Figure \ref{four_example_curves} along with the corresponding unimodular triangulations.  We will affectionately refer to these four types as honeycomb, Mickey Mouse, one-bridge, and two-bridge, names that are clear when contracting all the infinite edges and leaves.  


\begin{proposition}\label{types}
Up to homeomorphism, there are precisely five connected trivalent planar graphs of genus three having no leaf edges. 
Of these, precisely four of the homeomorphism classes are realizable as skeletons of smooth tropical plane quartic curves.
\end{proposition}

\begin{proof}
Let $G$ be a trivalent planar graph of genus three. By Euler's formula, together with a counting argument, $G$ has four vertices and six edges. Suppose first that two of the vertices of $G$  are connected by more than one edge. Then by the trivalent condition, they must be connected by two edges, and the only possible graphs are the Mickey Mouse graph, the one-bridge graph, and the two-bridge graph.  

Next, assume that $G$ does not have multiple edges. Since the genus is three, it is not possible for each vertex to have a loop attached to it, so by trivalency there is at least one vertex that is connected to each of the others by an edge. Now, each vertex may have a loop attached to it, or it may be connected to other vertices, giving rise to the honeycomb graph and the \emph{lollipop graph}, pictured at the far right of Figure~\ref{genus_three_graphs}.

We now show that the lollipop graph is never the skeleton of a smooth tropical plane quartic curve.  If a graph arises from a unimodular triangulation, then every bridge that separates the graph to topologically non-trivial components is dual to a line segment in the subdivision that separates the polygon into two smaller polygons, each of which contains at least one interior lattice point.  Inspection reveals that there are six possible such splits for $T$: they are the segments joining the pairs of points $\{(0,2),(3,0)\}$, $\{(0,2),(3,1)\}$, $\{(2,0),(0,3)\}$, $\{(2,0),(1,3)\}$, $\{(2,2),(1,0)\}$, and $\{(2,2),(0,1)\}$.  However, no three of these splits may appear in the same triangulation due to intersections.  It follows that the lollipop graph, which has three bridges, does not arise as the skeleton of a smooth tropical plane quartic curve.
\end{proof}

\begin{figure}[hbt]
\label{fig:lollipop} 
\begin{center}
\includegraphics[
height=.5in
]{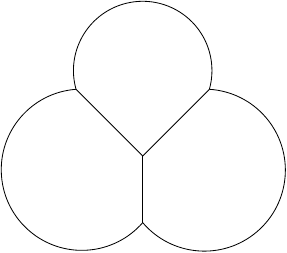}
\,\,\,\,\,\,\,\,\,\,
\includegraphics[
height=.4in
]{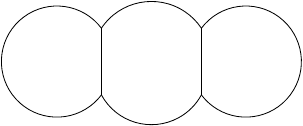}
\,\,\,\,\,\,\,\,\,\,
\includegraphics[
height=.4in
]{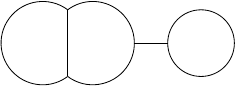}
\,\,\,\,\,\,\,\,\,\,
\includegraphics[
height=.4in
]{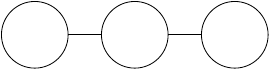}
\,\,\,\,\,\,\,\,\,\,
\includegraphics[
height=.5in
]{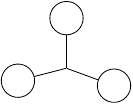}
\end{center}
\caption{The five homeomorphism classes of trivalent connected leafless genus $3$ graphs, the first four of which occur in smooth tropical plane quartics.}
\label{genus_three_graphs}
\end{figure}
 
 
 \begin{remark}
By a computation in TOPCOM \cite{Rambau:TOPCOM-ICMS:2002}, there are precisely 7422 regular unimodular triangulations of the standard triangle $T$ with vertices $(0,0), (4,0), (0,4)$.  Of the resulting 1278 $S_3$-equivalence classes (with respect to the natural action of $S_3$ on $T$),  573 are honeycomb, 450 are Mickey Mouse, 225 are one-bridge, and 30 are two-bridge.
\end{remark}

\end{subsection}

\begin{subsection}{Divisor theory on metric graphs}   
Next, we briefly recall the theory of divisors on tropical curves (see \cite{GK,MZ} for more details).  A \emph{metric graph} is a graph $\Gamma$ together with a function $\ell:E(\Gamma)\rightarrow\mathbb{R}_{>0}\cup\{\infty\}$ assigning a length to each edge, with infinite lengths only allowed on edges containing a one-valent vertex. 
A \emph {divisor} $D$ on a metric graph $\Gamma$ is an element of the free abelian group on $\Gamma$, i.e., a finite formal sum of points 
$$
D = a_1\cdot p_1 +\ldots +a_n\cdot p_n.
$$
The {\em degree} of $D$ is $ {\rm deg}(D):=a_1 + \ldots + a_n$, and $D$ is said to be \emph {effective} if all the coefficients are non-negative. 
We will sometimes refer to the points of a divisor as (virtual) \emph{chips}, so that $D$ has $a_i$ chips on $p_i$ for each $i$.

\medskip

A \emph {rational function} on $\Gamma$ is a continuous piecewise linear function with integer slopes. To a rational function $f$ we associate a divisor $\text{div}(f)$ whose degree at each point $p$ is the sum of the outgoing slopes of $f$ at $p$. A divisor of this form is called \emph {principal}. Note that the degree of a principal divisor is always zero. 
Two divisors $D$ and $D'$ are said to be {\em linearly equivalent} if $D-D' = \text{div}(f)$ for some rational function $f$. A divisor $D$ has \emph {rank} $r$ if for every effective divisor $E$ of degree $r$, $D-E$ is linearly equivalent to an effective divisor, and $r$ is the largest integer with this property.

\medskip

A \emph{canonical divisor} on $\Gamma$ is any divisor linearly equivalent to
$$K_\Gamma:=\sum_{v\in \Gamma}(\text{valence}(v)-2)\cdot v. $$
Its degree is $2g-2$, where $g$ is the \emph {genus} of $\Gamma$, i.e., the dimension of $H_1(\Gamma,{\mathbf R})$.

\medskip

The {tropical Riemann-Roch theorem} \cite{BN} asserts that if $\Gamma$ is a metric graph of genus $g$ and $D$ is any divisor on $\Gamma$, then 
\[
r(D) + r(K_\Gamma - D) = {\rm deg}(D) + 1 - g.
\]

\medskip

The skeleton of a smooth plane tropical curve can be viewed as a metric graph by letting the length of an edge $e$ be its {\em lattice length} (see \cite[Definition 3.4]{KMM} or \cite[Section 6.2]{BPR}).
In the same way, we can (more generally) consider \emph {any} connected compact subset of a smooth tropical plane curve as a metric graph.
The lattice lengths of the bounded edges in a tropical plane curve are {\bf not} determined solely by the subdivided Newton polygon; they depend on
the actual coefficients of a defining tropical polynomial.

\medskip

The definition of the metric on a smooth tropical plane curve is justified in part by the following folklore result.  It deals with a subgraph $\Sigma$ of a tropical curve $C$, which is a collection of edges in $C$ together with the adjacent vertices.

\begin{lemma} \label{lemma:folklore}
Let $C$ be a smooth tropical plane curve and let $L_1, L_2$ be tropical lines. Suppose that $\Sigma$ is a connected subgraph of $C$ that contains the skeleton and the stable intersections $L_1\cdot C$ and $L_2\cdot C$. Then  $L_1\cdot C$ and $L_2\cdot C$ are linearly equivalent as divisors on $\Sigma$.
\end{lemma}

\begin{proof}
For $i=1,2$, let $f_i$ be the pullback to $C$ of the defining piecewise-linear equation of the tropical line $L_i$, and let $\phi_i$ be the rational function on $\Sigma$ obtained by restricting $f_i$ to $\Sigma$. 

By \cite{AR}, the Weil divisor $\text{div}(f_i)$ on $C$ is exactly the stable intersection of $L_i$ and $C$ counted with multiplicity. For a point $x$ in the interior of $\Sigma$, we see from the definitions that  $\text{div}(\phi_i)(x)$ coincides with $\text{div}(f_i)(x)$.
For $x$ on the boundary of $\Sigma$, the difference between $\text{div}(f_i)(x)$ and $\text{div}(\phi_i)(x)$ equals the outgoing slope of $f_i$ on the infinite edge of $C$ emanating from $x$. If $x$ is on an edge pointing in the $(1,1)$ direction, then since we are assuming that the intersection of $L_i$ and $C$ is contained in $\Sigma$, it follows that $x$ is not in the region spanned by the two rays of $L_i$ pointing in the $(-1,0)$ and $(0,-1)$ directions. Therefore, the slope of $f_i$ on the infinite edge emanating from $x$ is $1$. Similarly, if $x$ is on an infinite edge pointing in a different direction, the slope of $f_i$ is zero on that edge.

Therefore,  $\text{div}(\phi_i)$ is the divisor obtained from $\text{div}(f_i)(x)$ by subtracting a chip from every boundary point of $\Sigma$ on an infinite edge pointing in the $(1,1)$ direction. 
It follows that $\phi_1 - \phi_2$ is a rational function on $\Sigma$ whose associated principal divisor is exactly $C\cdot L_1 - C\cdot L_2$, as required.

\end{proof}

\end{subsection}

\end{section}

\begin{section}{The 7 bitangent lines to a smooth tropical plane quartic}
A classical fact in algebraic geometry is that a smooth plane quartic curve has $28$ bitangent lines. In this section, we will adapt a classical proof of that fact to the tropical world, and show that every smooth plane tropical curve admits $7$ bitangent lines up to a certain equivalence relation.    

\begin{definition} \label{def:bitangent}
A tropical line $L$ is said to be \emph{bitangent} to a smooth tropical plane quartic curve $C$ if $L\cap C$ consists of two components each with stable intersection multiplicity $2$, or one component with stable intersection multiplicity $4$.  A tropical bitangent line is called \emph{skeletal} if its intersection with the tropical curve is contained in the skeleton. 
\end{definition}

An outline of a classical proof that there are 28 bitangent lines is as follows. Every smooth algebraic plane quartic curve $X$ is canonically embedded, so any line section is a canonical divisor. It follows that bitangent lines correspond bijectively to linear equivalence classes\footnote{Since $X$ is not hyperelliptic,
one can in fact remove the phrase ``linear equivalence classes'' here.} of effective divisors $D$ of degree two such that $2D$ is a canonical divisor. Such divisors are called \textit{effective theta characteristics}.  The set of all theta characteristics (effective or not) is canonically a torsor for ${\rm Jac}(X)[2]$, which has order $64$. 
Via a non-trivial analysis of bilinear forms in characteristic $2$ as in \cite[Ch. VIII.2]{DO}, one proves that there are exactly 28 effective theta characteristics, obtaining the desired result.  


\medskip

One can define theta characteristics for a metric graph in the same way:

\begin{definition} 
A \emph {theta characteristic} on a metric graph $\Gamma$ is a linear equivalence class of divisors $D$ such that $2D$ is linearly equivalent to the
canonical divisor on $\Gamma$.
A theta characteristic is called {\em effective} if one can choose $D$ to be an effective divisor.
\end{definition}

Zharkov proves in \cite{Zh} that a metric graph $\Gamma$ of genus $g$ has $2^g$ theta characteristics, exactly $2^g - 1$ of which are effective.
The fact that there are $2^g$ theta characteristics (instead of $2^{2g}$ as in the classical case) comes from the fact that ${\rm Jac}(\Gamma)$ is a {\bf real} (rather than complex) torus of dimension $g$.  In both the classical and tropical situations the set of theta characteristics is naturally a torsor for the 2-torsion in the Jacobian.  Zharkov's proof that all but one of the theta characteristics on $\Gamma$ is effective does not seem to have an algebraic analogue; for
the reader's convenience we summarize some of the ideas behind Zharkov's proof in Lemma~\ref{thetaAlgorithm} below.

\medskip

In particular, Zharkov's theorem tells us that an abstract metric graph $\Gamma$ of genus 3 has exactly $7$ effective theta characteristics.  However, unlike the classical case, it is not obvious (and in fact not true) that bitangent lines to a smooth tropical plane quartic $\Gamma$
are in bijection with effective theta characteristics.  In fact, it is not even completely obvious {\em a priori} how to define a tropical bitangent line!  
With our definition (which seems to be the only reasonable one), there are examples of smooth tropical plane quartics with infinitely many bitangent lines;
see Example~\ref{genus3example}. (This is reminiscent of Vigeland's
example of cubic surfaces with infinitely many lines rather than 27 \cite{Vi}; a possible relationship between tropical plane quartic curves and tropical cubic surfaces warrants further investigation.) Moreover, unlike the classical case, it appears to be subtle to prove that for every effective theta characteristic $[D]$, there is a tropical line $L$ and a divisor $D' \sim D$ such that $L\cdot C = 2D'$.
We will prove this via a case-by-case analysis of the different possible smooth tropical plane quartic curves.

\medskip

First, however, we show that while tropical plane quartics are not canonically embedded, it is still true that every line section belongs to the class of the canonical divisor.  We thank Yang An for suggesting the following proof:

\begin{lemma} \label{lineSectionsAreCanonicals}  
For a smooth tropical plane quartic curve $C$ and a tropical line $L$, the stable intersection divisor $L\cdot C$ is canonical.
\end{lemma}

 \begin{proof}
By Lemma~\ref{lemma:folklore}, any two divisors obtained as line sections of $C$ are linearly equivalent.   
Now let $L$ be a tropical line, and let $D$ be the divisor $L\cdot C$. By tropical Bezout's Theorem \cite{FirstSteps}, the degree of $D$ is $4$, and by the tropical Rieman-Roch theorem it is canonical if and only if its rank is $2$. 
Let $\Gamma$ be a compact subset of $C$ containing the support of $D$, considered as a metric graph.
Choose a loopless model $G$ for $\Gamma$, which is obtained from $\Gamma$ by adding two-valent vertices in the middle of edges so that $G$ contains no loops.  Let $E$ be any effective divisor of degree $2$ supported on the vertices of $G$. 

\medskip

If $E$ consists of two distinct vertices $u$ and $v$, let $L'$ be a tropical line passing through them. Since $u$ and $v$ are vertices, they are contained in the stable intersection $L'\cdot C$ (even when the intersection is not transversal). By the argument above, $L\cdot C$ is linearly equivalent to $L'\cdot C$, hence $D$ is equivalent to a divisor containing $E$. 

\medskip

Otherwise, $E$ consists of a double point at some vertex $w$ of $G$. Let $L$ be a tropical line with a vertex at $w$. Suppose first that, locally at $w$, $\Gamma$ is parallel to $L$, i.e. $\Gamma$ is trivalent at $w$ with edges pointing towards  $(-1,0), (0,-1), (1,1)$.  Since $C$ is not a line, at least one of these edges is finite and the stable intersection contains a point at $w$ and  at the vertices at the end of those edges. By firing the complement of a neighborhood of $w$, we obtain a divisor containing at least two chips at $w$.
If, on the other hand, $L$ is not parallel to $\Gamma$ at $w$, the stable intersection is easily seen to have multiplicity two at $w$.

\medskip
 
Since the vertices of $G$ are a rank-determining set by \cite{Luo}, the rank of $D$ is at least $2$, and therefore $D$ is canonical. 
 \end{proof}

Next, we describe Zharkov's algorithm for finding the theta characteristics of a metric graph $\Gamma$ (see \ref{genus3example} for an example).
Place a nonzero $(\mathbb{Z}/2\mathbb{Z})$-flow on the graph, namely, choose a subset of the edges so that around each vertex, the number of selected edges equals 0 modulo 2. Let $S$ be the support of the flow.  Place arrows on $\Gamma-S$ indicating movement away from $S$, with multiple arrows allowed on different portions of the same edge. Now, at the points where these arrows meet, place a number of chips equal to $\#\{\text{incoming edges}\}-1$.

\medskip

Zharkov proves:
 
\begin{lemma} \label{thetaAlgorithm}  
Let $\Gamma$ be a metric graph of genus $g$.  Then:
\begin{enumerate}
\item Each of the $2^g - 1$ divisors constructed as above is an effective theta characteristic.  
\item Different choices of $(\mathbb{Z}/2\mathbb{Z})$-flows yield non-linearly equivalent divisors. 
\item The remaining theta characteristic on $\Gamma$ is non-effective.
\end{enumerate}
\end{lemma} 


We refer the reader to \cite{Zh} for an explicit description of the non-effective theta characteristic, which will not be needed for this paper.

 \begin{example}\label{genus3example}  Consider the genus $3$ metric graph $\Sigma$ in Figure \ref{currentexample}.  (We have not included lengths, since it will not affect the combinatorics of this example.)  Place a nonzero $(\mathbb{Z}/2\mathbb{Z})$-flow on it, and let $P+Q$ be the corresponding effective theta characteristic for $\Sigma$.

We then consider a tropical curve $C$ which has $\Sigma$ as its minimal skeleton.  Drawing a line through $P$ and $Q$ does indeed give a bitangent to $C$, as shown in Figure \ref{currentexample} (showing the tropical curve without its exterior branches).  One component of the intersection, containing $P$, is a line segment with stable tropical intersection number $2$.  The other component, the point $Q$, is a tropical intersection of multiplicity $2$.  Hence $L\cdot C=2P+2Q$.
  
Note that the bitangent line in Figure \ref{currentexample} can be translated horizontally while remaining bitangent, leaving $P$ alone but moving $Q$.  This gives an infinite family of bitangents whose intersection divisors with $C$ are linearly equivalent.  This behavior can be seen already in the metric graph $\Sigma$:  there are effective divisors linearly equivalent to $P+Q$, namely those leaving $P$ fixed and moving $Q$ along the edge containing it.
 \end{example}
 
\begin{figure}[hbt]
\begin{center}
\includegraphics[
height=2in
]
{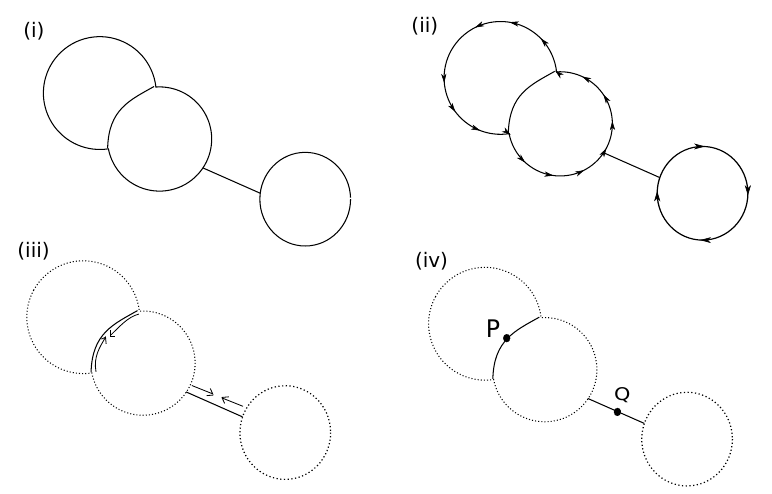}
\includegraphics[
height=1.7in
]
{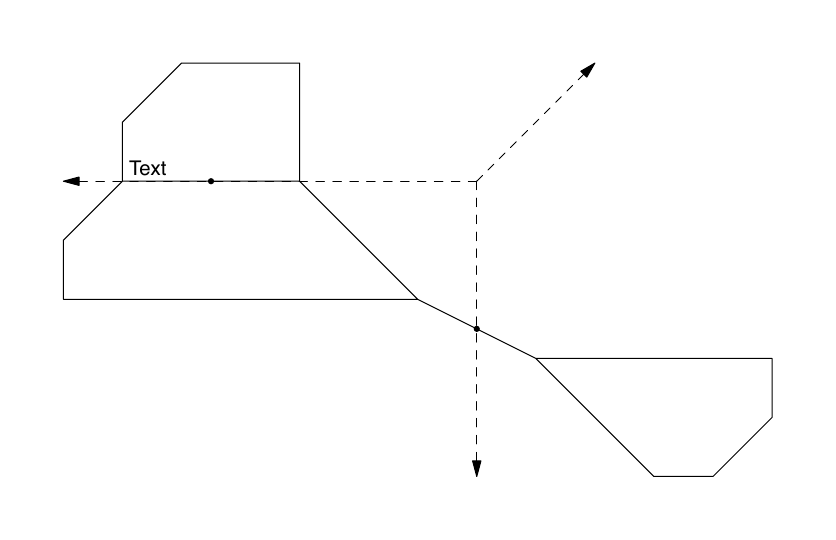}
\end{center}
\caption{The construction of a theta characteristic, and a corresponding bitangent line.}
\label{currentexample}
\end{figure}

In Appendix \ref{section:theta_chars}, we use Zharkov's algorithm to compute the $7$ effective theta characteristics for each of the first four graphs in Figure \ref{genus_three_graphs}.  The computation depends slightly on the relative lengths of certain edges, but is made easier by the fact that smooth plane quartic curves are not hyperelliptic (Theorem~\ref{hyperelliptic_theorem}).

\medskip

Next, we show that each effective theta characteristic on a smooth tropical plane quartic curve is represented by a bitangent line:
 \begin{proposition}\label{keyproposition}  Let $P+Q$ be an effective theta characteristic for a smooth tropical plane quartic curve $C$, and let $L$ be a tropical line connecting $P$ and $Q$.  Then $L$ is a bitangent  line to $C$.
 \end{proposition}

\begin{proof}
First, if the stable tropical intersection $L\cdot C$ does not contain $P$ or $Q$, we can move the divisor $L\cdot C$ within $L\cap C$ to a linearly equivalent  divisor $P+Q+R+S$, with all four points still in $L\cap C$.
We wish to show that $P+Q+R+S=2P+2Q$.

\medskip

By definition of a theta divisor we know that $2P+2Q$ is canonical, and by Lemma \ref{lineSectionsAreCanonicals}, $L\cdot C$ is canonical for any line $L$.  Thus $2P+2Q\sim P+Q+R+S$, and so $P+Q\sim R+S$.  There are three cases we have to deal with.

\begin{itemize}

\item[(i)] $P+Q$ is {\bf rigid} (i.e., not linearly equivalent to any other effective divisor). 
In this case, it follows immediately that $P+Q=R+S$.
 
 \smallskip
 
 \item[(ii)]  One or both of $P$ and $Q$ is {\bf flexible}, i.e., can be moved independently of the other while maintaining linear equivalence.
    In this case, we claim that it suffices to show that the intersection multiplicity of $L$ with $C$ at each flexible point is 2.  Indeed, if the multiplicity at $P$ is 2, and $Q$ cannot move independently, then $L\cdot C = 2P + Q + S \sim K$ together with $2P + 2Q\sim K $ implies that $Q\sim S$. Since we assumed that $Q$ cannot move independently, we obtain $Q=S$, namely $L\cdot C = 2P + 2Q$. 
    
 \smallskip
    
    \noindent To see that the multiplicity at each flexible point is 2,  we notice that, by \cite[Corollary 4.7]{BN}, a chip can only move independently when it is on a bridge edge.  This can only occur for the one-bridge and the two-bridge curves (cf. Figure \ref{four_example_curves}). The slopes of the bridge edges in these curves are always either $-2$ or $-\frac{1}{2}$. By Appendix \ref{section:theta_chars}, the two-bridge curve has at most four theta characteristics with a chip on a bridge, and the the one-bridge curve has as most three such theta characteristics (depending on the edge lengths). Moreover, the chip on a bridge with slope $-2$ is always obtained by intersection with the edge of a tropical line pointing in the $(-1,0)$ direction, and the chip on an edge with slope $-\frac{1}{2}$ is obtained by intersection with an edge pointing in the $(0,-1)$ direction. Either way, the intersection is of multiplicity 2.
    
 \smallskip
 
    
  \item[(iii)]  $P+Q$ is linearly equivalent to effective divisors obtainable only by moving the two points in tandem.
By Lemma \ref{twoChips} below, this can only occur when $P$ and $Q$ form a 2-edge-cut.  
By the classification in Appendix \ref{section:theta_chars}, this happens only for the Mickey Mouse and two-bridge combinatorial types.
 
  \begin{figure}[hbt]
\begin{center}
\includegraphics[
height=1.4in
]{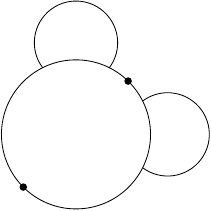}
\includegraphics[
height=1.4in
]{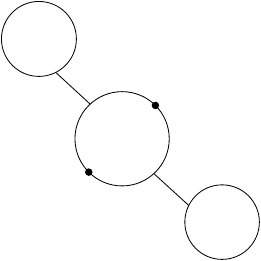}
\end{center}
\caption{The divisor $D_1$ on the Mickey Mouse graph, and the divisor $D_2$ on the two-bridge graph.}
\label{figure:mickey_bad}
\end{figure}

\smallskip

\noindent We begin with the Mickey Mouse graph. From the combinatorial type of $C$, we know that the corresponding Newton subdivision has exactly two edges between the interior points $(1,1)$, $(2,1)$, and $(1,2)$.  By symmetry, we may assume that $(1,1)$ is connected to the two other points, 
giving a Newton subdivision of the form shown on the left side of Figure \ref{figure:mickey_bad2}.  The diagonal edge from $(1,1)$ to $(2,2)$ is present because an edge must separate $(2,1)$ from $(1,2)$ from being joined by an edge.  The simple cycle $\gamma\subset C$ containing the support of the divisor $D_1$ is determined by which other points connect to $(1,1)$ in the triangulation, with all possibilities for the edges and a corresponding cycle shown in Figure \ref{figure:mickey_bad2}.  The relevant points are labelled by $i$ with $1\leq i\leq 9$, and the edges are labelled by $e_i$, where $e_i$ is dual to the edge connecting point $i$ to $(1,1)$.  Although not all triangulations will have all of the edges, each triangulation will have some subset of them making up the cycle $\gamma$.

 \begin{figure}[hbt]
\begin{center}
\includegraphics[
height=1.4in
]{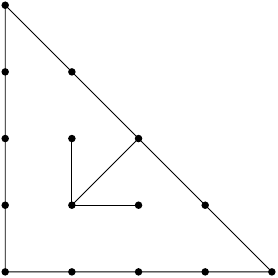}
\includegraphics[
height=1.6in
]{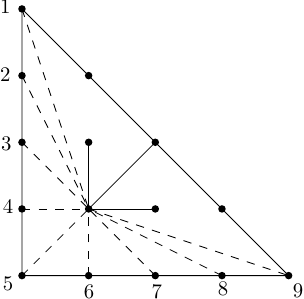}
\includegraphics[
height=1.6in
]{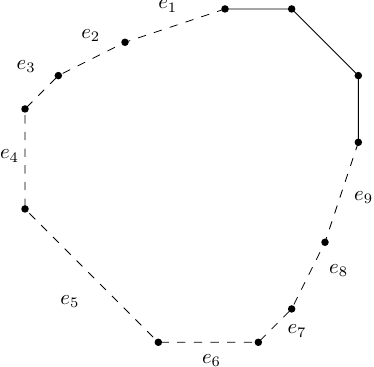}
\end{center}
\caption{Triangulations for the Mickey Mouse combinatorial type, and the middle cycle.}
\label{figure:mickey_bad2}
\end{figure}

\smallskip 
 
\noindent Let $L_E$ be a Euclidean line with slope $1$ passing through the midpoint of the edge of $C$ dual to the segment connecting $(1,1)$ and $(2,2)$.  The line $L_E$ must pass through exactly one more point of the cycle $\gamma$, and this point must in fact be contained in $e_4\cup e_5\cup e_6$; this is because the edges $e_1$, $e_2$, and $e_3$ have slopes $\geq 0$ and $\leq 1$, and the edges $e_7$, $e_8$, and $e_9$ have slopes $\geq 1$ , making it impossible for $L_E$ to pass through them due to the position of the edges.
 
\smallskip
 
\noindent We claim that the tropical line $L$ defined by placing the vertex at the point $p_E=L_E\cap (e_4\cup e_5\cup e_6)$ is bitangent to $C$.  It will certainly pass through the edge  of $C$ dual to the segment connecting $(1,1)$ and $(2,2)$ with multiplicity $2$.  We must show that the other component of $C\cap L$ also has multiplicity $2$.  We can handle this in several cases:
 \begin{enumerate}
 \item If  $p$ is the interior of $e_4$, then $C\cap L$ contains the vertical line segment below $p$ on $e_4$, which will have multiplicity $2$.  A similar argument holds if $p$ is in the interior of $e_6$.
  \item If $p$ is in the interior of $e_5$, then $p$ is an isolated intersection point of $C$ and $L$, and will have multiplicity $2$ due to slope considerations.
  \item If $p=e_4\cap e_5$, then that component of $C\cap L$ is a horizontal ray in the direction $(-1,0)$.  A small perturbation will put us in one of the previous two cases (or give two intersection points), meaning it is an intersection of multiplicity $2$.  A similar argument holds if $p=e_5\cap e_6$.
   \item If $p=e_4\cap e_6$ (meaning there is no $e_5$), a small perturbation will give either a vertical line segment contained in $e_4$, a horizontal line segment contained in $e_6$, a diagonal line segment dual to the edge connecting points $4$ and $6$, or a pair of points.  This too will be an intersection of multiplicity $2$.
 \end{enumerate}
Hence $L$ is indeed a bitangent line to $C$.  To verify that it corresponds to the theta characteristic $D_1$, note that it must correspond to \emph{some} theta characteristic, and $D_1$ is the only one with both points on the middle cycle.

\smallskip

\noindent The proof for the two-bridge case is nearly identical to the Mickey Mouse case. Due to the combinatorial type, we know there are no edges between the three interior points in the triangulation, and by symmetry we may assume the point $(1,1)$ corresponds to the middle cycle.  This means there is an edge connecting $(1,1)$ to $(2,2)$.  The other points that $(1,1)$ may connect to are the points $1,\ldots, 9$ as in the previous proof.  If we consider a Euclidean line through the midpoint of the edge corresponding to the edge connecting $(1,1)$ to $(2,2)$, we find once again that it must pass through a point in $e_4\cup e_5\cup e_6$.  Letting $L$ be the tropical line having a vertex at that point, we once again argue that it is bitangent to $C$, and find that it corresponds to the theta characteristic $D_2$.


 \end{itemize}

  
 


\end{proof}

\begin{lemma}\label{twoChips}
Let $P,Q$ be two distinct points of $\Gamma$ such that $P+Q$ is not rigid, and neither $P$ nor $Q$ is flexible (as defined in Proposition \ref{keyproposition}).
Then $\Gamma\setminus\{P,Q\}$ is disconnected.
\end{lemma}
\begin{proof}
Let $\phi$ be a non-constant rational function such that $P+Q +\text{div}(\phi)$ is effective, and let $M$ be the global maximum of $\phi$. Then $\text{div}(\phi)<0$ on any boundary point of $\phi^{-1}(M)$ of valency greater than $1$. Therefore, the global maximum of $\phi$ is obtained at $P$ and $Q$, and any other point where the global maximum is obtained is an interior point of $\phi^{-1}(M)$. Moreover, since $\text{div}(\phi)$ is $-1$  at $P$ and $Q$, the slope of $\phi$ is $-1$ on a single edge leaving $P, Q$,  and $0$ on the others. 

Let $x$ be a point on an edge leaving $P$ where the slope is $0$, and $y$ a point on an edge leaving $Q$ where the slope is $-1$. Then $x$ is in $\phi^{-1}(M)$, and $y$ is not, so any path between $x$ and $y$ passes through the boundary of $M$, namely through either $P$ or $Q$. Therefore, by removing $P$ and $Q$, the graph becomes disconnected.
  
\end{proof}

Since we have seen (cf. Example~\ref{genus3example}) that a smooth tropical plane quartic curve can have infinitely many bitangent lines, we wish to define a notion of
equivalence for tropical bitangents so that if one bitangent can be moved to another while preserving the bitangency condition, then the two are equivalent. 
Formally, what works is the following:

\begin{definition} \label{def:equivalent_bitangents}
Two bitangent lines are \emph {equivalent} if they correspond to linearly equivalent theta characteristics.
\end{definition}

Note that any bitangent line is equivalent to a skeletal bitangent line. To see this, suppose that a tropical line intersects $C$ at the points $P$ and $Q$ with multiplicity $2$, and that either of them is on the infinite branches and not on the skeleton $\Sigma$ of the curve. Then, since the infinite branches are leaf edges, the theta divisor $P+Q$ is linearly equivalent to a divisor $P'+Q'$ on the skeleton. By Proposition \ref{keyproposition}, the line passing through $P'+Q'$ is a bitangent line.

\medskip

Combining this observation with  Lemma \ref{thetaAlgorithm} and Proposition \ref{keyproposition}, we obtain the following result:

\begin{theorem} 
Every smooth tropical plane quartic admits precisely 7 bitangent lines up to equivalence. 
\label{bitangent_theorem}
\end{theorem}

 \end{section}
 
 \begin{section}{Smooth tropical plane quartics are not hyperelliptic}
 
As is well known, a smooth algebraic plane quartic curve is never hyperelliptic. In this section, we prove a tropical analogue of this statement. This is a useful result in computing theta characteristics for genus $3$ graphs, as it put restrictions on possible edge lengths (cf. Appendix \ref{section:theta_chars}).  

\medskip

We begin by recalling the definition of a hyperelliptic graph.

\begin{definition}\label{hyperelliptic_definition}
A metric graph $\Gamma$ is \emph {hyperelliptic} if it has a divisor of degree $2$ and rank $1$.
\end{definition}

We will use the following necessary condition for hyperellipticity.

\begin{lemma}\label{twoBridge}
Let $\Gamma$ be a metric graph with a $2$-edge-cut. Then $\Gamma$ is hyperelliptic only if the two edges of the cut are of equal length.
\end{lemma}

\begin{proof}
Let $e_1$ and $e_2$ be the two edges of the cut, and let $\Gamma'$ and $\Gamma''$ be the two connected components of $\Gamma\setminus\{e_1\cup e_2\}$. By the metric graph analogue of \cite[Corollary 5.10]{BN2} (which is proved in exactly the same way as in {\em loc. cit.}, using \cite[Lemma 3.11]{Ch} instead of \cite[Lemma 5.7]{BN2}),
we may assume that $\Gamma$ is leafless with no bridge edges.
Let $D = p+q$ be a divisor of degree $2$ and rank $1$ on $\Gamma$. Since $D$ has positive rank, we may assume that $p$ is on the intersection between $e_1$ and $\Gamma'$.

\medskip

We will begin by showing that $q$ is on $e_1\cup e_2$. Assume by contradiction that it is not, and choose some point $x$ in the interior of $e_2$. Then one easily checks that $D$ is $x$-reduced, contradicting the fact that $D$ has rank $1$.
Moreover, $q$ has to be on the intersection between $\Gamma'$ and $e_2$. Otherwise, choose a point $y$ in $\Gamma'$. Then again, $D$ is $y$-reduced,  and we arrive at a contradiction.

\medskip

Now, move $p$ and $q$ along the cut while preserving linear equivalence, until one of them reaches $\Gamma''$. By the same argument as before (switching the roles of $\Gamma'$ and $\Gamma''$, and of $p$ and $q$ if necessary), both $p$ and $q$ must reach $\Gamma''$ together. It follows that $e_1$ and $e_2$ have the same length.
\end{proof}

Let $C$ be a smooth tropical plane quartic. 
As discussed in Subsection \ref{quartics}, $C$ is trivalent of genus $3$, and its skeleton must have one of the first four combinatorial types depicted in Figure \ref{genus_three_graphs}.  According to \cite{Ch}, out of these four, only the ones shown in Figure \ref{figure:combinatorial_types} have hyperelliptic metric
realizations.  The following theorem shows that none of these realizations occur for the skeleton of a smooth tropical plane quartic curve.

\begin{figure}[hbt]
\centering
\includegraphics[scale=0.5]{comb_type_mickey}
\includegraphics[scale=0.5]{comb_type_one_bridge}
\includegraphics[scale=0.5]{comb_type_two_bridge}
\linebreak \includegraphics[scale=0.3]{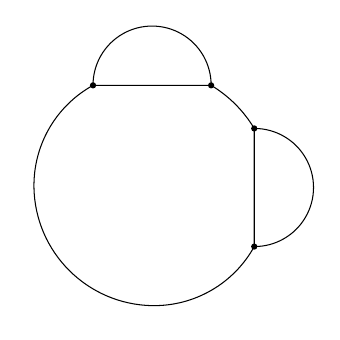}
\includegraphics[scale=0.3]{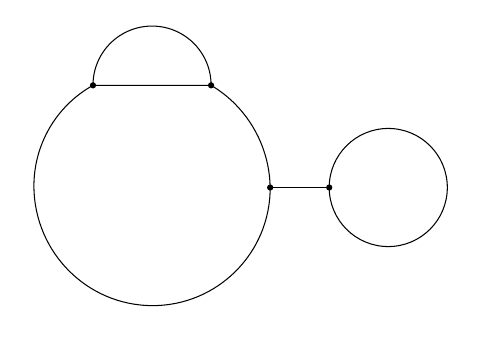}
\includegraphics[scale=0.3]{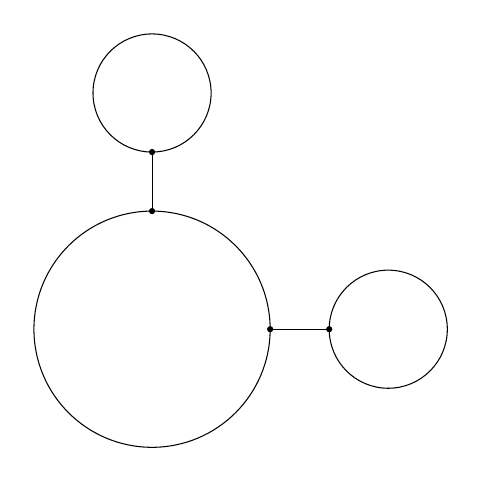}
\caption{The three combinatorial types of graphs realizable as trivalent hyperelliptic metric graphs of genus $3$, and the same graphs with lengths making them non-hyperelliptic.
}
\label{figure:combinatorial_types}
\end{figure}


\begin{theorem}\label{hyperelliptic_theorem}
Every smooth tropical plane quartic curve is non-hyperelliptic.
\end{theorem}

\begin{proof}

By Lemma \ref{twoBridge}, whenever a hyperelliptic metric graph contains a $2$-edge-cut, that cut consists of two edges of equal length. As we will see, this never occurs for smooth tropical plane quartic curves.

\smallskip

Consider a unimodular triangulation of the standard triangle with vertices $(0,0), (4,0), (0,4)$, corresponding to a tropical curve $C$. There are four possible combinatorial types of the skeleton $\Gamma$ of $C$, namely the first four pictured in Figure \ref{genus_three_graphs}.  The first one (the ``honeycomb'' graph) is never hyperelliptic \cite{Ch}, regardless of the length of the edges.  The other three can be hyperelliptic if given the appropriate lengths.  We will show that such lengths are never achieved by $\Gamma$. 

\smallskip

Without loss of generality, we may assume that the point $(1,1)$ corresponds to the large cycle $\gamma$ in the middle as in the second row of Figure \ref{figure:combinatorial_types}.
Each edge emanating from $(1,1)$ in the triangluation corresponds to an edge of the cycle. Figure \ref{figure:triangle_cycle} shows all possibilities for such edges and the corresponding cycle. We know that the edge from $(1,1)$ and $(2,2)$ is present in the triangulation:  some edge must separate $(2,1)$ and $(1,2)$, and any other separating edge prevents $(1,1)$ from being the middle cycle.

\smallskip

\begin{figure}
\centering
\includegraphics[scale=0.5]{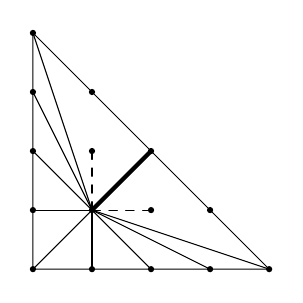}
\includegraphics[scale=0.5]{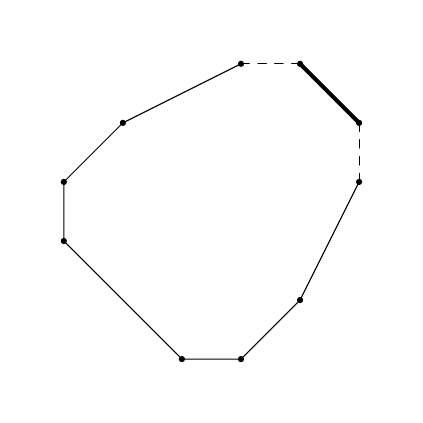}
\caption{Edges coming from $(1,1)$, and the corresponding cycle.}
\label{figure:triangle_cycle}
\end{figure}

In the three combinatorial types in question, the 2-edge-cut consists of certain edges of the cycle $\gamma$.  The first edge, which we will denote $e_1$, is dual to the edge connecting (1,1) with (2,2) in the triangulation. The other edge of the cut, denoted $e_2$, comes from the edges in the triangulation emanating from (1,1) with angles strictly between $\frac{\pi}{4}$ and $2\pi$. 
While a triangulation does not uniquely determine the lattice lengths of the edges in a tropical curve, we claim that the data of the directions of the edges suffices to conclude that $e_2$ must be longer than $e_1$.  We will show this for each of the three combinatorial types separately.  Let $\ell_1$ and $\ell_2$ be the lengths of the edges $e_1$ and $e_2$.

\begin{figure}
\centering
\includegraphics[scale=1]{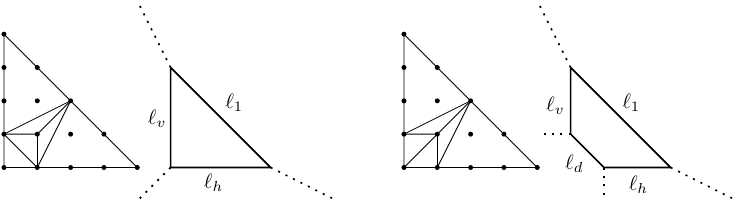}
\caption{Possibilities for the middle cycle of the two-bridge graph, with the corresponding subdivision structure}
\label{figure:two_bridge_cycle}
\end{figure}

\begin{enumerate}
\item Mickey Mouse. In this case, the cycle $\gamma$ decomposes as $e_1$-$e_h$-$e_2$-$e_v$, where $e_h$ and $e_v$ are horizontal and vertical edges with lengths $\ell_h$ and $\ell_v$, respectively.  Consider the line segments in $e_2$ that are dual to edges in the triangulation connecting $(1,1)$ with points on the horizontal edge of the Newton polygon.  The sum of the horizontal widths of these segments must be at least the sum of the horizontal widths of $e_1$ and $e_h$: otherwise the cycle $\gamma$ would not be closed.  Since these line segments in $e_2$ have integer slopes, each of them has lattice length equal to horizontal width.  The same holds for $e_1$ and $e_h$, implying $\ell_2\geq \ell_1+\ell_h>\ell_1$.

\item  One-bridge.  In this case, without loss of generality the cycle $\gamma$ decomposes as $e_1$-$e_h$-$e_2$, where $e_h$ is a horizontal edge of length $\ell_h$.  The same argument from the Mickey Mouse case holds, giving $\ell_2\geq \ell_1+\ell_h>\ell_1$.

\item Two-bridge.  In this case, $e_1$ and $e_2$ form the whole cycle $\gamma$.  Since edges must separate $(1,1)$ from $(1,2)$ and $(2,1)$ in the triangulation, there are very few possibilities for the form of $e_2$.  Either $e_2$ consists of a vertical segment and a horizontal segment with lengths $\ell_v=\ell_1$ and $\ell_h=\ell_1$, which implies $\ell_2=2\ell_1$; or $e_2$ consists of a vertical, a diagonal, and a horizontal segment with lengths $\ell_v$, $\ell_d$, and $\ell_h$.  In this latter case, $\ell_v=\ell_h$ and $\ell_d=\ell_1-\ell_v$, so $\ell_2=\ell_v+\ell_d+\ell_h=\ell_v+(\ell_1-\ell_v)+\ell_v=\ell_v+\ell_1>\ell_1$.  Both of the cases are illustrated in Figure \ref{figure:two_bridge_cycle}.
\end{enumerate}

This proves the claim, and hence the theorem.


\end{proof}


 \end{section}

 \appendix
 \begin{section}{Theta characteristics for genus $3$ graphs}\label{section:theta_chars}

 In the proof of Proposition \ref{keyproposition}, we made use of knowledge about the $7$ theta characteristics of each of the four graphs arising as a skeleton of a smooth plane quartic curve.  In Figure \ref{figure:28_theta_characteristics}, we illustrate all the theta characteristics for the four types of genus 3 graphs relevant to us. Each graph has the support of a nonzero $({\mathbf Z}/2{\mathbf Z})$-flow in bold, and in most cases the corresponding theta characteristic is illustrated as a pair of circular points.  In the cases where edge lengths might change the combinatorial position of the points of the theta characteristic, the other possibility is illustrated by a pair of crosses.  There are degenerate cases where one of these moves to a vertex, but this will not affect our arguments.  We have also taken advantage of Theorem \ref{hyperelliptic_theorem}, which allows us to assume asymmetry for the middle cycle in the second, third, and fourth columns.

\begin{figure}[p]
\centering
\includegraphics[scale=0.75]{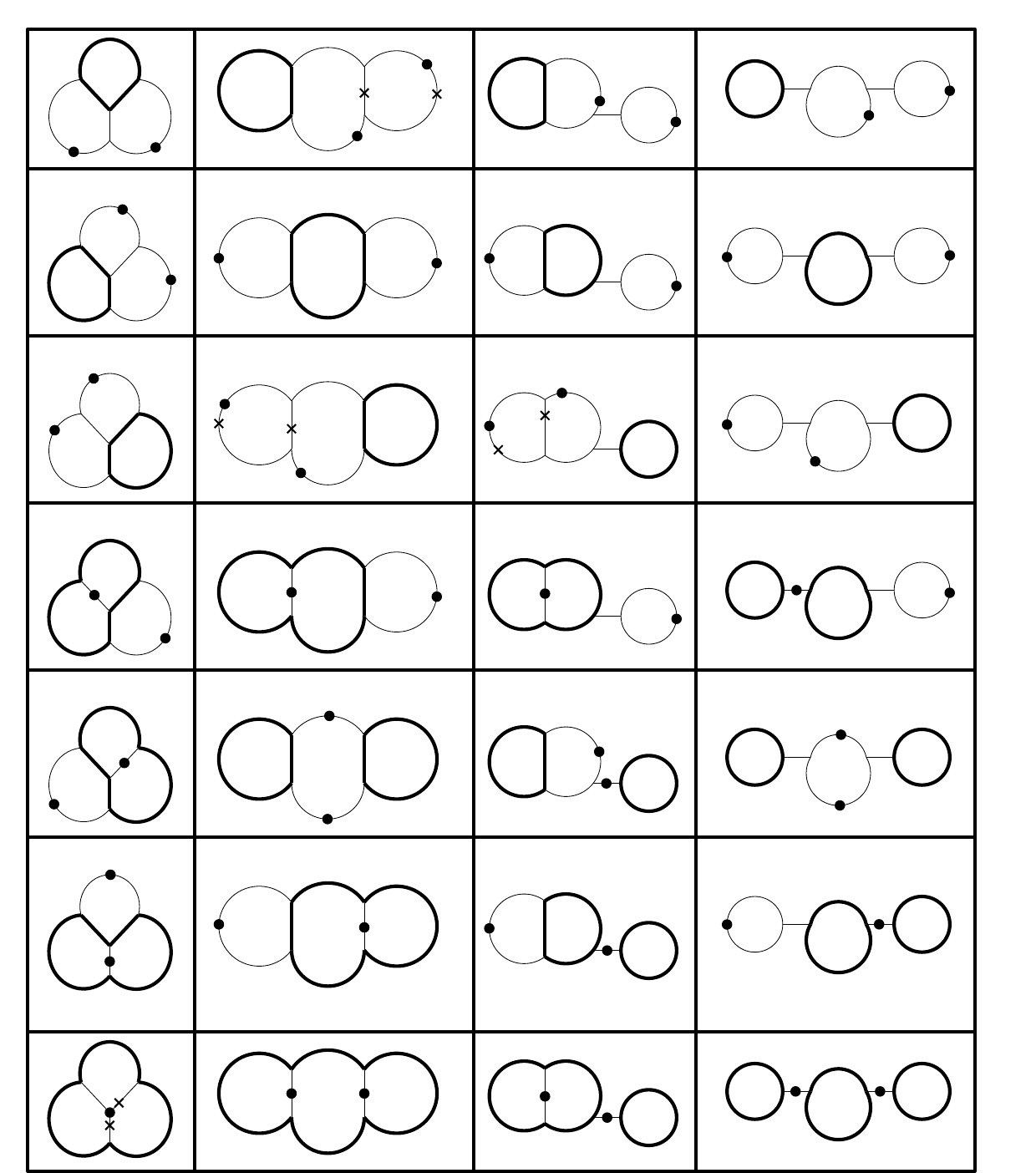}\caption{The $7$ theta characteristics for each type of graph.}
\label{figure:28_theta_characteristics}
\end{figure}

Labeling the columns 1, 2, 3, and 4 and the rows A, B, C, D, E, F, and G, we have the following classification of these $28$ theta characteristics by cases (i), (ii), and (iii) of Proposition \ref{keyproposition}:
\begin{itemize}
\item[(i)] The $20$ not falling into cases (ii) or (iii).
\item[(ii)] The $6$ theta characteristics 3E, 3F, 3G, 4D, 4F, 4G.
\item[(iii)] The $2$ theta characteristics 2E, 4E.
\end{itemize}

 \end{section}
  
\bibliographystyle{plain}
\bibliography{blmprbib}
 
 \

\noindent Matt Baker - mbaker@math.gatech.edu 
 
 \noindent
School of Mathematics, Georgia Institute of Technology 
 
  \noindent
Atlanta, GA 30332-0160 (USA)
 
 \
  
 \noindent
 Yoav Len - yoav@math.uni-sb.de 
 
 \noindent
 Fachrichtung Mathematik, Universit{\"a}t des Saarlandes,
 
  \noindent
Postfach 151150, 66041 Saarbr{\"u}cken, Germany

 \

 \noindent
Ralph Morrison -  ralphmo@kth.se

 \noindent
Department of Mathematics, KTH

  \noindent
SE-100 44  Stockholm, Sweden
 
 \

 \noindent
Nathan Pflueger -  pflueger@math.harvard.edu

 \noindent
Department of Mathematics, Harvard University
 
  \noindent
1 Oxford Street, Cambridge, MA 02138 (USA)

 \

 \noindent
Qingchun Ren -  qingchun.ren@gmail.com

 \noindent
Google Inc.
	 
  \noindent
1600 Amphitheatre Parkway, Mountain View, CA 94043 (USA)
\end{document}